\documentclass[12pt]{amsart}

\usepackage{amscd, amsmath, amssymb, amsthm, amsfonts, amsxtra, 
enumerate, latexsym, tabularx, setspace}
\usepackage[all]{xy}
\usepackage[dvips]{graphicx,color}

\def\FF{{\mathbb{F}}}

\def\m{{\mathfrak{m}}}

\def\PP{{\mathbb{P}}}

\def\R{{\mathbb{R}}}

\def\ZZ{{\mathbb{Z}}}

\def\Der{{\mathrm{Der}}}

\def\Frac{{\mathrm{Frac}}}

\def\Hom{{\mathrm{Hom}}}

\def\Proj{{\mathrm{Proj\; }}}

\def\Spec{{\mathrm{Spec\; }}}

\def\der{\partial}

\theoremstyle{plain}
\newtheorem{thm}{Theorem}[section]
\newtheorem*{theorem}{Theorem}

\newtheorem{lem}[thm]{Lemma}

\theoremstyle{definition}
\newtheorem{dfn}[thm]{Definition}

\theoremstyle{remark}

\newtheorem*{prob}{Problem}

\title{Globally F-regular $F$-sandwiches of degree $p$ of a projective space}
\author{Tadakazu Sawada}
\address{Department of General education, Fukushima National College of Technology, 
30 Aza-Nagao, Kamiarakawa, Iwaki-shi, Fukushima 970-8034, Japan}
\email{sawada@fukushima-nct.ac.jp}

\begin{document}

\begin{abstract}
We prove that globally F-regular $F$-sandwiches of degree $p$ of a projective space 
are toric varieties. 
\end{abstract}

\maketitle
\markboth{Tadakazu Sawada}{Globally F-regular $F$-sandwiches of $\PP^n$ of degree $p$}
\section*{Introduction}
We work over an algebraically closed field $k$ of positive characteristic $p$. Let $X$ 
be a variety over $k$. If an iterated Frobenius morphism $F^e:X\rightarrow X$ factors 
as $X\rightarrow Y\rightarrow X$, we say $Y$ is a Frobenius sandwich of $X$. If $e=1$, 
Frobenius sandwich $Y$ is called an $F$-sandwich. Given a variety, it is natural to ask 
what kinds of varieties appear as Frobenius sandwiches. From the view point of Frobenius 
splitting, we have considered the following problem in \cite{HS}: 
\begin{prob}
Given a globally F-regular variety $X$, classify globally F-regular Frobenius sandwiches of $X$.
\end{prob}
In \cite{HS} and \cite{S}, we have classified globally F-regular $F$-sandwiches of the projective 
plane and Hirzebruch surfaces. $F$-sandwiches are constructed by glueing quotients of affine 
patches by a rational vector field. The classifications have been achieved by explicit calculations 
of coordinate changes. In this paper, we consider globally F-regular $F$-sandwiches of a projective 
space $\PP^n$. The following is the main result of this paper. 
\begin{theorem}
Globally F-regular $F$-sandwiches of degree $p$ of $\PP^n$ are toric varieties. 
\end{theorem}
For the proof, we give a description of $F$-sandwiches as $\mathrm{Proj}$ of the constant ring 
of the homogeneous coordinate ring of $\PP^n$ by a global section of $T_{\PP^n}$. Using that 
description, we show that globally F-regular $F$-sandwiches of degree $p$ of $\PP^n$ are toric 
varieties without tedious calculations of coordinate changes. 

In Section 1, we review generalities on Frobenius sandwiches and globally F-regular varieties. In 
Section 2, we give a description of global sections of tangent bundle $T_{\PP^n}$ as a derivation 
over the homogeneous coordinate ring of $\PP^n$. In Section 3, we give the proof of the main 
theorem. 

\section{Frobenius sandwiches and Globally F-regular varieties}
In what follows, we do not distinguish the absolute Frobenius morphism and the relative one, 
since we work over the algebraically closed field. See \cite{H} for the definitions of those 
Frobenius morphisms. 

First we review generalities on Frobenius sandwiches. 

\begin{dfn}[\cite{HS}, \cite{E}]
Let $X$ be a smooth variety over $k$. A normal variety $Y$ is an {\it $F^e$-sandwich} of $X$ 
if the $e$-th iterated relative Frobenius morphism of $X$ factors as 
$$\xymatrix
{X \ar[rr]^{F^e_{\rm rel}} \ar[rd]_{\pi}  &                  & X^{(-e)} \\
                                          & Y \ar[ru]_{\rho} &
}$$
for some finite $k$-morphisms $\pi : X \rightarrow Y$ and $\rho : Y \rightarrow X^{(-e)}$, which are 
homeomorphisms in the Zariski topology. The Frobenius sandwich $Y$ is of {\it degree $p$} if the 
degree of the morphism $\pi : X \rightarrow Y$ is $p$. 

A {\it $1$-foliation} of $X$ is a saturated $p$-closed subsheaf $L$ of the tangent bundle $T_X$ 
closed under Lie brackets, where $L$ is said to be $p$-closed if it is closed under $p$-times iterated 
composite of differential operators. 
\end{dfn}

Let $X$ be a smooth variety over $k$, and $K(X)$ be the function field of $X$. A rational vector field 
$\delta \in \Der_k K(X)$ is {\it $p$-closed} if $\delta^p=\alpha \delta$ for some $\alpha \in K(X)$. 
Then there are one-to-one correspondences among the followings: 
\begin{itemize}
\item $F$-sandwiches of degree $p$ of $X$; 
\item invertible $1$-foliations of $X$; 
\item $p$-closed rational vector fields of $X$ modulo an equivalence $\sim$.
\end{itemize}
\noindent(1) Rational vector fields and $F$-sandwiches: Let $\delta,\delta' \in \Der_k\, K(X)$. We 
denote $\delta \sim {\delta}'$ if there exists a non-zero rational function $\alpha \in K(X)$ such that 
$\delta = \alpha {\delta}'$. We can easily check that $\sim$ is an equivalence relation between rational 
vector fields. Let $\{U_i = \Spec\, R_i\}_i$ be an affine open covering of $X$. Given a $p$-closed rational 
vector field $\delta \in \Der_k\, K(X)$, we have a quotient variety $X/{\delta}$ defined by glueing 
$\Spec\, R_i^{\delta}$, where $R_i^{\delta}=\{r\in R_i\,|\,\delta(r)=0\}$, and a quotient map 
$\pi_{\delta} : X \rightarrow X/{\delta}$ induced from the inclusions $R_i^{\delta} \subset R_i$. Then 
$X/\delta$ is an $F$-sandwich of degree $p$ with the finite morphism $\pi_{\delta} : X \rightarrow X/{\delta}$. 

\noindent(2) Rational vector fields and $1$-foliations: A rational vector field $\delta \in \Der_k\, K(X)$ 
is locally expressed as $\alpha \sum f_i \der/\der s_i$, where $s_i$ are local coordinates, $f_i$ are regular 
functions without common factors, and $\alpha \in K(X)$. The divisor $\mathrm{div}(\delta)$ associated to 
$\delta$ is defined by glueing the divisors $\mathrm{div}(\alpha)$ on affine open sets. We see that the 
multiplication map $\cdot\,\delta :\mathcal{O}_X(\mathrm{div}(\delta))\rightarrow T_X$ defined by 
$h/\alpha \mapsto h\sum f_i \der/\der s_i$ induces an inclusion $\mathcal{O}_X(\mathrm{div}(\delta))\subset T_X$, 
and $\mathcal{O}_X(\mathrm{div}(\delta))$ is an invertible $1$-foliation of $X$. See \cite{RS}, \cite{E}, \cite{Hirokado}, 
\cite{S} for more details of the correspondences. 

Next we recall the definition of global F-regularity. 
\begin{dfn}[\cite{Smith}]
A projective variety over an F-finite field is {\it globally F-regular} if it admits some section ring that is F-regular. 
\end{dfn}
For example, projective toric varieties are globally F-regular. In particular, projective spaces are globally 
F-regular. See \cite{Smith}, \cite{SS} for more examples and the general theory of globally F-regular 
varieties. In our situation, global F-regularity has a closed connection with splitting of Frobenius sandwiches. 

\begin{lem}[\cite{SS}, \cite{S} Lemma 1.2]\label{split_vs_F-reg}
Let $X$ be a globally F-regular variety over $k$ and $Y$ be an $F^e$-sandwich of $X$ with the finite morphism 
$\pi : X \rightarrow Y$ through which the Frobenius morphism of $X$ factors. Then $Y$ is globally F-regular if and 
only if the associated ring homomorphism $\mathcal{O}_Y \rightarrow \pi_{\ast} \mathcal{O}_X$ splits as an 
$\mathcal{O}_Y$-module homomorphism. 
\end{lem}

\section{Global sections of $T_{\PP^n}$}
Let $R=k[X_0,\ldots ,X_n]$, $X=\PP^n=\Proj R$, $S=H^0(X,\mathcal{O}_X(1))=k\langle X_0,\ldots ,X_n\rangle$, 
$D_i={\partial}/{\partial X_i}$, and $D_E=\sum_{i=0}^n X_i D_i \in \Der_k R$. There exists an isomorphism $f$ between 
$k$-modules $H^0(X,T_X)$ and $\bigoplus_{i=0}^{n} S D_i/\left(D_E\right)$ defined by the composition 
$$H^0(X,T_X) \rightarrow \displaystyle\bigoplus_{i=0}^n S e_i/\left(\sum_{i=0}^n X_i e_i\right)\rightarrow 
\bigoplus_{i=0}^{n} S D_i/\left(D_E\right),$$
where the first map is induced by the Euler sequence and the second one is defined by $e_i\mapsto D_i$. 
Let $x_j=X_j/X_i$, $R_i=R_{(X_i)}=k[x_0, \ldots ,x_n]$, $U_i=D_+(X_i)=\Spec R_i$, and 
$d_j={\partial}/{\partial x_j}\in \Der_k R_i$. Then 
$$X_sD_i({X_t}/{X_i})=
\left\{\begin{array}{ll}
-{X_sX_t}/{X_i^2}=-x_sx_t & (s\not=i)\\[1mm]
-{X_t}/{X_i}=-x_t & (s=i)
\end{array}\right.$$
and 
$$X_sD_j({X_t}/{X_i})=
\left\{\begin{array}{ll}
0 & (t\not=j)\\[1mm]
{X_s}/{X_i}=x_s & (t=j)
\end{array}\right.$$
for $j\not=i$. 
Hence  
$$X_sD_i=
\left\{\begin{array}{ll}
-x_s\displaystyle\sum_{t\not=i} x_td_t & (s\not=i)\\[1mm]
-\displaystyle\sum_{t\not=i} x_td_t& (s=i)
\end{array}\right.$$
and $X_sD_j=x_sd_j$ for $j\not=i$ on an open set $U\subset U_i$. 
We define the {\it restriction map} $\varphi_U:\bigoplus_{i=0}^{n} S D_i/\left(D_E\right) \rightarrow H^0(U,T_X)$ 
by $\overline{X_sD_i}\mapsto -x_s\sum_{t\not=i} x_td_t\ (s\not=i)$, $\overline{X_iD_i}\mapsto -\sum_{t\not=i} x_td_t$, and 
$\overline{X_sD_j}\mapsto x_sd_j\ (j\not= i)$ for the open set $U\subset U_i$. 
Then we have a commutative diagram 
$$
\xymatrix{
H^0(X,T_X) \ar[d]_{\rho_{XU}} 
& \displaystyle\bigoplus_{i=0}^{n} S e_i/\left(\sum X_i e_i\right) \ar[l] \ar[r] \ar[d] 
& \displaystyle\bigoplus_{i=0}^{n} SD_i/\left( D_E\right) \ar@/^25mm/[lld]_{\varphi_U}\\
H^0(U,T_X) & H^0(U, \mathrm{Coker} (\mathcal{O}_X\!\rightarrow \!\mathcal{O}_X(1)^{n+1})) \ar[l] & 
}$$
where $\rho_{XU}$ is the restriction map of $T_X$. 
Therefore we have the following Lemma: 

\begin{lem}\label{comm}
For any open set $U\subset U_i$, 
we have a commutative diagram
$$
\xymatrix{
H^0(X,T_X) \ar[r]^(.45){f} \ar[d]_{\rho_{XU}} & \displaystyle\bigoplus_{i=0}^{n} SD_i/\left(D_E\right) \ar[ld]^{\varphi_U}\\
H^0(U,T_X) & 
}$$
\end{lem}

\begin{lem}\label{loc-glob}
Let $\overline{D}\in \bigoplus_{i=0}^{n}SD_i/\left(D_E\right)$ and $\delta \in H^0(X,T_X)$ be the corresponding 
global section of $T_X$. Then $\Proj R^D\cong X/\delta$. 
\end{lem}
\begin{proof}
Let $F\in R$ be a homogeneous polynomial such that $F\in R^D$. Replacing $F$ by $X_i^p F$, we may assume 
that $D_+(F)\subset U_i$. We can easily check that $(R_{(F)})^D=(R^D)_{(F)}$. For an open set 
$D_+(F)\subset \Proj R^D$, we have $D_+(F)\cong \Spec (R^D)_{(F)}\cong \Spec (R_{(F)})^D$. On the other hand, 
$X/\delta$ is defined by glueing $D_+(F)/\delta\cong \Spec (R_{(F)})^{\delta}\cong \Spec (R_{(F)})^D$, since 
$\delta=D$ on $D_+(F)\subset U_i$ by Lemma~\ref{comm}. Therefore we have $\Proj R^D\cong X/\delta$. 
\end{proof}

We define the $p$-th composition $\overline{D}^p$ of $\overline{D}\in \bigoplus_{i=0}^{n}SD_i/\left(D_E\right)$ by 
$\overline{D^p}$. We say that $\delta \in H^0(X,T_X)$ (resp. 
$\overline{D}\in \bigoplus_{i=0}^{n}SD_i/\left(D_E\right)$) is $p$-closed if $\delta^p=\alpha \delta$ (resp. 
$\overline{D}^p=\alpha \overline{D}$) for some $\alpha \in k$. 

\begin{lem}\label{p-closed}
Let $\overline{D}\in \bigoplus_{i=0}^{n}SD_i/\left(D_E\right)$ and $\delta \in H^0(X,T_X)$ be the corresponding global 
section of $T_X$. Supposed that $D\not=0$. If $\delta$ is $p$-closed, then $D$ is also $p$-closed. 
\end{lem}
\begin{proof}
Let $K=\Frac R$. If $\delta$ is $p$-closed, then $\overline{D}$ is also $p$-closed. Hence 
$\overline{D^p}=\overline{D}^p=\alpha \overline{D}=\overline{\alpha D}$ for some $\alpha\in k$ and 
$D^p-\alpha D+\beta D_E=0$ for some $\beta\in k$. Since $D\in \bigoplus_{i=0}^{n}SD_i$, we have 
$D^{p+1}-\alpha D^2+\beta D=(D^p-\alpha D+\beta D_E)\circ D=0$. This means that 
$t^{p+1}-\alpha t^2+\beta t\in K^D[t]$ is divided by the minimal polynomial $\mu_{D} (t) \in K^D[t]$ of $D$, where we 
consider $D$ as a $K^D$-linear map $K\rightarrow K$. Since 
$\mu_{D} (t)=t^{p^m}+\sum_{i=0}^{m-1} a_i t^{p^i}$ with $a_i\in R^D$ by \cite{AA} Lemma 2.4, we have 
$\mu_{D}(t)=t^p+a_0t$. In particular, $D^p+a_0D=0$. Since $D\in \bigoplus_{i=0}^{n}SD_i$, we see that $a_0\in k$. 
Therefore $D$ is $p$-closed. 
\end{proof}

\section{Globally F-regular F-sandwiches of degree $p$ of $\PP^n$}
We will use the following lemmas in the proof of the main result. 
\begin{lem}\label{cong}
Let $X$ be a smooth variety, $Y$ be a globally F-regular $F$-sandwich of degree $p$ of $X$ with the finite morphism 
$\pi : X \rightarrow Y$, and $L \subset T_X$ be the corresponding $1$-foliation. Then we have
$$\Hom_{\mathcal{O}_Y}(\pi_*\mathcal{O}_X,\mathcal{O}_Y) \cong H^0(Y,\pi_{\ast}(L^{\otimes (p-1)}))= H^0 (X, L^{\otimes (p-1)}).$$
\end{lem}
\begin{proof}
See \cite{HS} Theorem 3.4.
\end{proof}

\begin{lem}\label{key}
Let $\m=(X_0, \ldots ,X_n)$ be the maximal ideal of $R=k[X_0,\ldots ,X_n]$, and $D \in \bigoplus_{i=0}^{n}SD_i$ be a 
$p$-closed derivation. Supposed that $D$ is not nilpotent. Then there exists $D'=\sum_{i=0}^n a_i D_i \in \Der_k R$ with 
$a_i \in \FF_p$ such that $R^{D} \cong R^{D'}$. 
\end{lem}
\begin{proof}
Since $D \in \bigoplus_{i=0}^{n}SD_i$ and $D$ is not nilpotent, we have $D^p=\alpha D$ for some $\alpha \in k^{\times}$. 
Replacing $D$ by $\alpha^{1/(1-p)} D$, we may assume that $D^p = D$. We define 
$\overline{D} \in \Der_k (\m/\m^2)$ by $\overline{D}(\overline{f})=\overline{D (f)}$. Since $D^p - D =0$, the minimal 
polynomial $\mu_{\overline{D}} (t) \in k[t]$ of $\overline{D}$ divides $t^p-t=t(t-1)(t-2) \cdots (t-(p-1))$. Hence $D$ is 
diagonalizable with eigenvalues $a_0, \ldots , a_n\in \mathbb{F}_p$. Let $Y_0, \ldots , Y_n$ be elements of $S$ such that 
$\overline{Y_0}, \ldots , \overline{Y_n} \in \m/\m^2$ are linearly independent eigenvectors of $\overline{D}$ corresponding 
to eigenvalues $a_0, \ldots ,a_n$, respectively. Then 
$\overline{D(Y_i)}=\overline{D}(\overline{Y_i})=a_i\overline{Y_i}=\overline{a_iY_i}$. Since $D \in \bigoplus_{i=0}^{n}SD_i$ and 
$Y_i\in S$, we have $D (Y_i)\in S$. Thus $D (Y_i)=a_i Y_i$. After a change of coordinates $X_i\mapsto Y_i$, we have 
$D = \sum_{i=0}^n a_i Y_i \der/\der Y_i$. This completes the proof. 
\end{proof}

\begin{lem}\label{toric}
Let $\overline{D}=\overline{\sum_{i=0}^n a_iX_iD_i}\in \bigoplus_{i=0}^n SD_i/\left(D_E\right)$ with $a_i\in\FF_p$. Supposed 
that $\overline{D}\not=0$. Then $\Proj R^D$ is a toric variety. 
\end{lem}
\begin{proof}
We refer to \cite{CLS} for the general theory of toric varieties. 

Let $\delta\in H^0(\PP^n,T_{\PP^n})$ be the corresponding global section of $T_{\PP^n}$. We have 
$\overline{D}=\overline{\sum_{i=1}^n(a_i-a_0)X_iD_i}$ and $a_i-a_0\not=0$ for some $i$. Replacing $a_1-a_0$ by $a_i-a_0$, 
and multiplying $\overline{D}$ by $(a_1-a_0)^{-1}$, we may assume that $\overline{D}=\overline{X_1D_1+\sum_{i=2}^na_iX_iD_i}$. 
Then we have $\delta|_{U_0}=\overline{D}|_{U_0}=x_1d_1+\sum_{i=2}^n a_ix_i d_i$, which is a description of $\delta$ as a rational 
vector field. 

Let $N = \ZZ^n$ be a lattice, $M$ be the dual lattice of $N$, and $\Sigma$ be the fan in $N\otimes \R$ corresponding to the 
projective space $\PP^n$. Let $N'=N+\ZZ\,\dfrac{1}{p}(1,a_2, \ldots , a_n)$ be an overlattice of $N$, and $M'$ be the dual 
lattice of $N'$. We have 
\begin{eqnarray*}
M' &=& (N')\spcheck \\
 &=& \left( \ZZ\dfrac{1}{p}(1,a_2,\ldots ,a_n) \oplus \ZZ\dfrac{1}{p}(1,a_2-p,\ldots ,a_n)\oplus \cdots \oplus \ZZ\dfrac{1}{p}(1,a_2,\ldots ,a_n-p) \right)\spcheck \\
 &=& \ZZ(p-a_2-\cdots-a_n,1,\ldots,1) \\
 & & \oplus \ZZ (a_2, -1, 0, \ldots ,0) \oplus \cdots \oplus \ZZ (a_n, 0, \ldots , 0, -1) \\
 &=& \ZZ (p,0,\ldots,0) \oplus \ZZ(a_2, -1, 0, \ldots ,0) \oplus \cdots \oplus \ZZ(a_n, 0, \ldots , 0, -1) \\
 &=& \left\{ (s_1, \ldots ,s_n) \in M \left| s_1 + \sum_{i=2}^{n}a_is_i \equiv 0 \mod p \right. \right\} \subset M. 
\end{eqnarray*}
Let $\sigma_i$ be the cone corresponding to $U_i=D_+(X_i)\subset \PP^n$. Since 
$(R_i)^{\delta}=R_i\cap K(\PP^n)^{\delta}=k[(\sigma_i)\spcheck \cap M']$, we see that $\PP^n/\delta$ is the toric variety whose 
corresponding fan is $\Sigma$ in $N'\otimes \R$. Since $\PP^n/\delta\cong \Proj R^D$ by Lemma~\ref{loc-glob}, $\Proj R^D$ 
is the toric variety. 
\end{proof}

\begin{thm}
Globally F-regular $F$-sandwiches of degree $p$ of $\PP^n$ are toric varieties. 
\end{thm}
\begin{proof}
Let $Y$ be a globally F-regular $F$-sandwich of degree $p$ of $\PP^n$ with the finite morphism $\pi : \PP^n \rightarrow Y$ 
through which the Frobenius morphism of $\PP^n$ factors, and let $L \subset T_{\PP^n}$ (resp. $\delta \in \Der_k\,K(\PP^n)$) 
be the corresponding $1$-foliation (resp. the $p$-closed rational vector field). Since the associated ring homomorphism 
$\mathcal{O}_Y \to \pi_*\mathcal{O}_{\PP^n}$ splits by Lemma~\ref{split_vs_F-reg}, there is a nonzero $\mathcal{O}_Y$-module homomorphism 
${\pi}_{\ast}\mathcal{O}_{\PP^n} \rightarrow  \mathcal{O}_Y$. By Lemma~\ref{cong}, $L^{\otimes (p-1)}$ has a nonzero global section, 
and so dose $L$. Let $\delta =\alpha \sum f_i \der/\der s_i$ be a local expression of $\delta$, where $s_i$ are local coordinates, 
$f_i$ are regular functions without common factors, and $\alpha \in K(X)$. Multiplying $\delta$ by a suitable rational function, 
we may assume that $\alpha$ are regular functions, since $\mathcal{O}_X(\mathrm{div}(\delta))\cong L$. Then $\delta$ is a global section 
of $T_X$. Let $\overline{D}$ be the corresponding element of $\bigoplus_{i=0}^n SD_i/\left(D_E\right)$. Since $\delta$ is $p$-closed, 
$D$ is also $p$-closed by Lemma~\ref{p-closed}. Then there exists $D'=\sum a_i X_i D_i\in \Der_k R$ with $a_i \in \FF_p$ such that 
$R^D\cong R^{D'}$ by Lemma~\ref{key}. We have $Y\cong \PP^n/\delta\cong \Proj R^D\cong \Proj R^{D'}$ by 
Lemma~\ref{loc-glob} and Lemma~\ref{key}. Therefore we see that the $F$-sandwich $Y$ is a toric variety by Lemma~\ref{toric}. 
\end{proof}


\end{document}